\documentclass[11pt,leqno,oneside,a4paper]{amsart}
\usepackage{amsmath}
\usepackage{amsthm}
\usepackage{amsfonts}
\usepackage{amssymb}
\usepackage{xcolor}
\usepackage{paralist}       % improved enumerate and itemize
\usepackage[mathcal]{euscript}

\usepackage{xspace}
\usepackage{mathtools}
\usepackage{hyperref}
\usepackage{natbib}
\hypersetup{
     breaklinks=true,
     colorlinks=true,
     linkcolor=teal,
     citecolor=violet}
\usepackage{graphicx}
\usepackage{todonotes}

\colorlet{darkgreen}{green!80!black}

\theoremstyle{plain}
\newtheorem{theorem}{Theorem}[section]
\newtheorem{lemma}[theorem]{Lemma}

\newtheorem{proposition}[theorem]{Proposition}
\newtheorem*{proposition*}{Proposition}

\theoremstyle{definition}

\newtheorem{definition}[theorem]{Definition}
\newtheorem{example}[theorem]{Example}

\numberwithin{theorem}{section}
\numberwithin{equation}{section}

\newenvironment{myproof}{
  \par\medskip\noindent
  \textit{Proof}.
}{
\newline
\rightline{$\square$}
}

\DeclarePairedDelimiter\abs{\lvert}{\rvert}
\DeclarePairedDelimiter\nrm{\|}{\|}
\DeclarePairedDelimiter\set{\lbrace}{\rbrace}

\DeclarePairedDelimiter\floor{\lfloor}{\rfloor}

\def\measurable{\mathcal{M}}

\newcommand{\N}{\mathbb{N}}
\newcommand{\R}{\mathbb{R}}

\newcommand{\lgn}{\mathcal L}

\def\ls{\lesssim}
\def\gs{\gtrsim}

\newcommand{\Wlgn}{W_{\kern-1.2pt\lgn}\kern0.7pt}

\newcommand{\pe}\relax
\newcommand{\pb}\relax
\newcommand{\emphm}\emph

\makeatletter
\DeclareRobustCommand\onedot{\futurelet\@let@token\@onedot}
\def\@onedot{\ifx\@let@token.\else.\null\fi\xspace}
 
\def\ie{i.e\onedot}

\def\ri{r.i\onedot} 
\makeatother

\makeatletter
\def\paragraph{\bigskip\@startsection{paragraph}{4}%
   \z@\z@{-\fontdimen2\font}%
   {\normalfont\bfseries}}
\makeatother

\title{On certain fine properties of sequence spaces}
\author{Hanuš Kameník}
\begin{document}

\maketitle

\textbf{Institution: }Charles University, Faculty of Mathematics and Physics, Department of Mathematical Analysis, Prague, Czech Republic

\textbf{e-mail address: }kamenik@karlin.mff.cuni.cz

\textbf{ORCID: }0009-0005-6357-5009



\section{Introduction}

The theory of function spaces and sequence spaces is one of the most rapidly developing parts of functional analysis. Many of its applications boil down to some properties of an operator on certain function space.

The most important motivation for studying function spaces is of course the transfer of regularity of solutions for partial differential equations. Traditionally the main role had Lebesgue spaces, but there are many examples, which show that the scale of Lebesgue spaces is not fine enough. Those problems require more finer or general scales of function spaces. 

For this reason many extensions and generalizations of Lebesgue spaces were introduced in the first part of the twentieth century. Notably the Zygmund classes of logarithmic and exponential type, the Lorentz spaces by introducing a second parameter and the Orlicz space depending on a Young function, which is a generalization of a power function governing the Lebesgue space. The construction of an Orlicz space is relatively simple by the usage of the Minkowski functional creating a Luxemburg norm, on the other hand the Lorentz space requires the symmetrization of a function into its non-increasing rearrangement. Both Lorentz and Orlicz spaces are rearrangement invariant, but as it turns out, the extensions are quite orthogonal and the two spaces only meet at the Lebesgue scale.

In 1980, Bennett and Rudnick~\cite{Ben:80} introduced the three-parameter scale of the~\textit{Lorentz–Zygmund spaces}, in order to exploit some properties of Orlicz spaces while retaining the Lorentz scale. Those spaces and their generalized version were extensively treated Opic and Pick \cite{Opick:99}. One of the many answers was the characterization of the intersection with Orlicz spaces. The paper takes into account only spaces above a non-atomic measure space, but the question for fully atomic spaces was left unanswered. 

This is where this paper hopes to fill the gap. The theory of function spaces over non-atomic measure has been treated in depth, the sequence counterpart has not been developed as much and there in not much literature diving very deep into the theory of sequence spaces. Some authors, notably G. Bennett and K.G. Grosse-Erdmann have participated in the development of sequence spaces \cite{Ben:05,Ben:06}. It is well known that the sequence counterpart is not just a simplification of the function theory, but, on the contrary, many methods often used for function spaces lose their effectiveness or could not be used at all. This forces us to use and develop new methods and approach some problems completely differently. The subject of sequence spaces was treated in several monographs, see e.g.~\cite{Lin:73,Gro:98} 

In this paper we answer, among others, the question where the Orlicz and the Lorentz--Zygmund sequence scales meet, see Theorem \ref{T: Charakterization-of-intersection}. Aside from this principal result, we prove a fair number of useful properties of general sequence spaces  and of the two very important scales, namely Lorentz--Zygmund spaces and Orlicz spaces, in particular. Our achievements include evaluation of the fundamental sequences of the spaces in question and an intensive study of their equivalent (quasi)-normability. We also establish their associate spaces, and more.

The paper can be rendered as a contribution towards the recently very dynamically developing field of sequence structures, which in turn reflects challenging contemporary problems whose solution might involve discrete inequalities and fine properties of sequences  (for an excellent survey for this direction of research, see~\cite{Pie:09}).

The paper is structured as follows. The next section is devoted to preliminary material on rearrangements, while the third section contains the foundations of Banach sequence spaces. The subsequent two sections are dedicated to a detailed study of Orlicz and Lorentz--Zygmunds sequence spaces, respectively. Our main result, that is, a detailed treatment of mutual relations between the latter two classes, is contained in the final section.

\section{Rearrangement and distribution function}

For the whole paper, we consider a measurable space \((\mathbb{N},\mathcal{P}(\mathbb{N}),m)\), where $m$ is an arithmetic measure. Note that $m$ is $\sigma$-finite and every sequence 
\[
\{a_n\}_{n=1}^\infty\colon \mathbb{N}\rightarrow\mathbb{R}^*
\]
is measurable. Let $\mathcal{M}$ be a set of all such sequences $\{a_n\}_{n=1}^\infty$.

The distribution function $a_*\colon [0,\infty)\rightarrow [0,\infty]$ of $a=\{a_n\}_{n=1}^\infty \in\mathcal{M}$ is defined as
\begin{equation*}
    a_*(\lambda) = m(\{n\in\mathbb{N}: |a_n|>\lambda\}).
\end{equation*}
Note that the distribution function $a_*$ has its values in whole numbers and infinity, it is non-increasing, right-continuous, and for any $\lambda>0$: $(\frac{a}{\lambda})_*(\tau) = a_*(\lambda \tau)$. 

The non-increasing rearrangement of the sequence $a$ is defined as:    
\begin{equation*}
    a_n^* = a^*(n-1)
\end{equation*}
where
\begin{equation*}
    a^*(t) = \inf\{\lambda\in[0,\infty):a_*(\lambda)\leq t\}.
\end{equation*}
and the maximal sequence $a^{**}$ as:
\begin{equation*}
    a^{**}_n = \frac{1}{n}\sum_{k=1}^n a_k^*
    \quad \text{for $n\in\N$.}
\end{equation*}

As in the theory of function spaces, a desired property is sub-additivity, unfortunately, the non-increasing rearrangement is not sub-additive, but similarly as in the continuous case we have a different, weaker version stated bellow, which is proven by using the standard dilatation argument and basic properties of rearrangement.
For $a,b\in \measurable$ it holds that
\begin{equation*}
    (a+b)^*_n \leq a^*_{\floor{\frac{n+1}{2}}} + b^*_{\floor{\frac{n+1}{2}}}.  
\end{equation*}
On the other hand, the sub-additivity holds for the operator of maximal sequence in the following sense.
For $a,b\in\measurable$ we have
\begin{equation*}
    (a+b)^{**}_n \leq a^{**}_n + b^{**}_n 
    \quad \text{for $n\in\N$.}
\end{equation*}

An important thing to note is that the rearrangement is not a function, but a sequence numbered from one. This results in only one singularity, which is at infinity and makes some spaces behave quite differently than their function counterpart. 

Here are some basic properties of non-increasing rearrangement used through the whole paper.

\begin{proposition} \label{P:Properties-of-rearrangement}
    Let $a, b\in\mathcal{M}, \{a^m\}_{m = 1}^\infty \subset \mathcal{M}$. Then
    \begin{enumerate} 
        \item[\textup{(i)}] if $a_n\geq a_{n+1} \geq 0$ for every $n\in\N$, then $a = a^*$, \label{P:rearrangement-of-non-increasing}
        \item[\textup{(ii)}]$a_n^* = \sup\{\min_{k\in E} \abs{a_k}: E\subset \N , m(E) = n\}$ 
            \quad for every $n\in\N$. \label{P:Properties-of-rearrangement-1}
        \item[\textup{(iii)}] If $0\leq a \leq b$, then $a^* \leq b^*$. \label{P:Properties-of-rearrangement-2}
        \item[\textup{(iv)}] If $0 \leq a^m \nearrow a$ point--wise, then $(a^m)^* \nearrow a^*$. \label{P:Properties-of-rearrangement-3}
    \end{enumerate}
\end{proposition}

\begin{myproof}
    (i) Let $n\in\N$, then, from the definition of the infimum, we find $\lambda_i \rightarrow a_n^*$ that satisfy $\forall i\in\N: \lambda_i\in \{\lambda\geq 0: a^*(\lambda)\leq n-1\}$. Because $a_*$ is right-continuous, one has $a_*(a_n^*) = \lim_{i\to \infty } a_*(\lambda_i)\leq n-1$, so
    \begin{equation*}
        m(\{k\in\N: a_k>a_n^*\})\leq n-1.
    \end{equation*}
    Therefore, from the monotonicity of $a$, we have
    \begin{equation*}
        \exists n_0\leq n-1, \forall k>n_0: a_k\leq a_n^*.
    \end{equation*}
    Especially, for $k=n$, we get $a_n\leq a_n^*$.

    From the monotonicity of $a$ we get $m(\{k\in\N: a_k>a_n\})\leq n-1$, therefore, it follows from the definition of $a_n^*$ that $a_n\geq a_n^*$.
    
    (ii) For $n\in \N$ set $c_n = \sup\{\min_{k\in E} \abs{a_k}: E\subset \N , m(E) = n\}$. To prove $a_n^* \leq c_n$ it is sufficient to show $a_*(c_n)\leq n-1$. For a contradiction, suppose that
    \begin{equation*}
        m\{i\in\N : \abs{a_i} > \sup_{m(E) = n} \min_{k\in E} \abs{a_k}\} \geq n.
    \end{equation*}
    Therefore, there is $F\subset\N, m(F) = n$ such that 
    \begin{equation*}
        \abs{a_i} > \sup_{m(E) = n} \min_{k\in E} \abs{a_k}
        \quad \text{for $i\in F$.}
    \end{equation*}
    Especially, we have
    \begin{equation*}
        \min_{i\in F}\abs{a_i} > \sup_{m(E) = n} \min_{k\in E} \abs{a_k},
    \end{equation*}
    which gives us the contradiction. The desired inequality now follows from the definition of $a^*$.

    Let $\varepsilon >0$. We find $E\subset \N$ such that $m(E) = n$ and $\min_{k\in E} \abs{a_k}> c_n -\varepsilon$. Then, we have
    \begin{equation*}
        m\{i\in \N : \abs{a_i} > c_n -\varepsilon\} \geq m\{i\in \N : \abs{a_i} \geq \min_{k\in E} \abs{a_k} \} \geq m(E) = n.
    \end{equation*}
    From this and the definition of $a^*$, we get $c_n - \varepsilon \leq a_n^*$. By taking $\varepsilon \rightarrow 0$, we obtain the second inequality.

    (iii) For $n\in \N$ from (ii), we immediately have
    \begin{equation*}
        a_n^* = \sup\{\min_{k\in E} \abs{a_k}: E\subset \N , m(E) = n\} \leq \sup\{\min_{k\in E} \abs{b_k}: E\subset \N , m(E) = n\} = b_n^*.
    \end{equation*}

    (iv) From the monotonicity of the convergence, (ii) and (iii) we get
    \begin{equation*}
        \lim_{m\to\infty} (a_n^m)^* = \sup_{m\in \N} \sup_{m(E)=n} \min_{k\in E} \abs{a_k^m} = \sup_{m(E)=n} \sup_{m\in \N}  \min_{k\in E} \abs{a_k^m}. 
    \end{equation*}
    To prove (iv), it is sufficient to show that
    \begin{equation*}
        \sup_{m\in \N}  \min_{k\in E} \abs{a_k^m} = \min_{k\in E} \abs{a_k}
        \quad \text{for $E\subset\N, m(E) = n$},
    \end{equation*}
    and then, by using (ii), we get (iv). Let $E\subset\N$ such that $m(E) = n$. The inequality $\sup_{m\in \N}  \min_{k\in E} \abs{a_k^m} \leq \min_{k\in E} \abs{a_k}$ follows immediately from $a^m \nearrow a$. Because $E$ is finite, there is $G\subset\N$ and $i\in E$ such that $G$ is infinite and $\min_{k\in \N} \abs{a_k^m} = a_i^m$ for every $m\in G$. From this we have
    \begin{equation*}
        \sup_{m\in \N}  \min_{k\in E} \abs{a_k^m} \geq \sup_{m\in G}  \min_{k\in E} \abs{a_k^m} \geq \sup_{m\in G} \abs{a_i^m} = a_i \geq \min_{k\in E} \abs{a_k}.
    \end{equation*}
    So the proof is complete.
\end{myproof}

\section{Banach sequence spaces and the fundamental sequence}

In this paper we use certain practical notation. To avoid confusion, we state the definitions here. If there exists a positive constant $C$ such that for two non-negative quantities $x$ and $y$ we have $x\leq Cy$, then we write $x \ls y$. If $x \ls y$ and $y\ls x$ we write $x\approx y$. 
If two functions $F,G\colon \, [0,\infty) \rightarrow [0,\infty)$ and some positive constants $c,C$ satisfy $F(ct) \leq G(t) \leq F(Ct)$ on some neighborhood of 0, then we denote this fact by $F \sim G$.

For $E\subset \N$, we define its characteristic sequence $\chi_E= \{\chi_E(n)\}_{n=1}^\infty$ as
\begin{equation*}
    \chi_E(n) = 
    \begin{cases}
    1
    &\text{if $n\in E$,}
    \\
    0
    &\text{otherwise.}
    \end{cases}
\end{equation*}

For $p\in (0,\infty]$ we define the Lebesgue (quasi-)norm as
\begin{equation*}
    \nrm{a}_p = \begin{cases}
                    \left(\sum_{n=1}^\infty \abs{a_n}^{p} \right)^\frac{1}{p} 
                    &\text{if $p<\infty$,}
                    \\
                    \sup_{n\in\N} \abs{a_n} 
                    &\text{if $p=\infty$,}
                \end{cases}
\end{equation*}
and the Lebesgue space $(\ell^p,\nrm{\cdot}_p)$ as
\begin{equation*}
    \ell^p = \{ a\in\measurable: \nrm{a}_p < \infty\}
\end{equation*}
with $\nrm{\cdot}_p$ as its (quasi)norm.

For the whole paper we consider the convention $\frac{1}{\infty}=0$. For $p\in [1,\infty]$ we denote by $p'$ the number satisfying $\frac{1}{p}+\frac{1}{p'} = 1$.

For better understanding of Lorentz--Zygmund and Orlicz spaces we first need to define general sequence space with special properties. All those properties come from the theory of Lebesgue spaces and are the foundation of its applications. If those properties hold, we can then use similar methods to prove different results for a wide variety of spaces, which are not Lebesgue spaces.

\begin{definition}\label{D:BSS}
    Let $\| \cdot \|_X:\, \mathcal{M} \rightarrow [0,\infty]$. If the following axioms hold:
    \begin{enumerate} [({P}1)]
        \item $\| \cdot \|_X$ is a norm and $\| a \|_X = \| \abs{a} \|_X$ whenever $a\in \mathcal{M}$, \label{D:BSS-abs}
        \item if $a,b \in \mathcal{M}$ are such that $0 \leq a \leq b$, then $\| a \|_X \leq \| b \|_X$,\label{D:BSS-monotonie}
        \item if $a^n, a \in \mathcal{M}$ and $0\leq a^n \nearrow a$ (point-wise), then $\| a^n \|_X \nearrow \| a \|_X$, \label{D:BSS-konvergence}
        \item $\| \chi_E \|_X < \infty$ whenever $m(E) <\infty$, \label{D:BSS-char. posloupnost}
        \item if $m(E) < \infty$, then there is a constant $C_E$ such that $\sum_{n\in E} \abs{a_n} \leq C_E \| a \|_X$ for all $a \in X$, \label{D:BSS-lok.l1}
    \end{enumerate}
    then we call such functional $\| \cdot \|_X$ a Banach sequence norm and the space $(X,\nrm{\cdot}_X)$, where $X= \set{a\in\measurable: \| \cdot\|_X < \infty}$ a Banach sequence space (BSS).
\end{definition}

As it turns out, these axioms are too strong to cover every desirable sequence space, especially the requirement for the defining functional to be a norm, thus we introduce a quasi-norm and with it a quasi-Banach sequence space as a weaker alternative. 

\begin{definition} \label{qBSS}
    Let $\| \cdot \|_X:\, \mathcal{M} \rightarrow [0,\infty]$ such that (P2)-(P5) hold and also:
    \begin{enumerate} [({Q}1)]
        \item $\nrm{a}_X=0$ iff $a=0$, $\nrm{ca}_X=\abs{c}\nrm{a}_X$ for every $c\in\R$ and there exists constant $C>0$, such that for every $a,b\in\measurable: \, \nrm{a+b}_X\leq C (\nrm{a}_X+\nrm{b}_X)$ and $\| a \|_X = \| \abs{a} \|_X$ whenever $a\in \mathcal{M}$.
    \end{enumerate}
    Then we call such functional a Banach sequence quasi-norm and the space$(X,\nrm{\cdot}_X)$, where $X = \set{a\in\measurable: \| \cdot\|_X < \infty}$ a quasi-Banach sequence space (qBSS).
\end{definition}

The weakening of the norm assumption to just a quasi-norm gained popularity quite recently and had been treated by A. Musilová, A. Nekvinda, D. Peša, and H. Turčinová ~\cite{Nek:24,Mus:25}.

One of many important properties of Lebesgue spaces is the rearrangement invariance. It is not hard to see that the Lebesgue norm depends only on the distribution function of the sequence. The spaces we consider in this paper are a generalization of Lebesgue spaces, so it is natural to require a similar property.

\begin{definition}\label{D:RI}
    Let $(X,\| \cdot \|_X)$ be a (quasi)-Banach sequence space. We will say that X is rearrangement-invariant (\ri) if $\| a \|_X = \| b \|_X$ whenever $a_* = b_*$, $a,b \in \mathcal{M}$.
\end{definition}

For sequences the axiom (P\ref{D:BSS-lok.l1}) is trivial, but for consistency with the theory of function spaces it is included in the definition. (P\ref{D:BSS-lok.l1}) follows from (P\ref{D:BSS-monotonie}), (P\ref{D:BSS-char. posloupnost}), positive homogeneity and \ri by the following calculation: 
\begin{align*}
    \sum_{n\in E} \abs{a_n} &\leq m(E) \max_{n\in E} |a_n| \leq m(E) \|a\|_\infty = m(E) a^*_1  = m(E) a_1^* \frac{\|\chi_{[1,1]}\|_X}{\|\chi_{[1,1]}\|_X} \\
    &= m(E) \frac{\|a_1^*\chi_{[1,1]}\|_X}{\|\chi_{[1,1]}\|_X} \leq m(E) \frac{\|a^*\|_X}{\|\chi_{[1,1]}\|_X} = m(E) \frac{\|a\|_X}{\|\chi_{[1,1]}\|_X}, 
    \quad \text{for $a\in\measurable$.}
\end{align*}

The structure of Banach sequence spaces and Banach function spaces is very similar, only differing on the underlying measure space, thus many classical results are applicable for our case. Very good foundation for general Banach function spaces is \cite[Chapter 1 and 2]{Ben:88} and in this paper it is used extensively.

If $X$ and $Y$ are qBSS such that $\nrm{\cdot}_X\ls \nrm{\cdot}_Y$, then we say $Y$ is embedded into $X$, writing $Y\hookrightarrow X$. If $Y \hookrightarrow X$ and $X \hookrightarrow Y$, we say the two spaces coincide, writing $X=Y$.

For the purposes of the theory of sequence spaces the classical dual is very complicated object, therefore we introduce an associate space as its alternative. It turns out that for some spaces its dual and its associate space are isomorphic. For our matter it is more important the second motivating property, which is the generalized H\"older inequality \eqref{E:Gener_Holder}. For $X$ a (quasi)-Banach space we define its associate norm as:
\begin{equation*}
    \| a\|_{X'} = \sup\left\{\sum_{n=1}^\infty \abs{a_n b_n}: b\in X, \|b\|_X \leq 1 \right\}
\end{equation*}
and the associate space $(X',\|\cdot\|_{X'})$ as $X'=\{ a\in \mathcal{M}: \| a\|_{X'}<\infty\}$.

Then for $a,b\in\mathcal{M}$ we have:
\begin{equation}\label{E:Gener_Holder}
    \sum_{n=1}^\infty a_n b_n \leq \|a\|_X \|b\|_{X'}.
\end{equation}

One of the most essential tools for research of \ri sequence spaces is the fundamental sequence. Its growth near infinity is a great indicator of many properties of the given space, in some cases it characterizes the space completely (see Proposition \ref{P:embeding-and-fundamental-sequence}). 

\begin{definition}\label{D:Fundamental sequence}
    Let $(X,\| \cdot \|)$ be a rearrangement-invariant sequence space, then for $n\in \N_0$ we define
    \begin{equation*}
        \varphi_X (n) = \| \chi_{E_n} \|_X 
        \quad \text{for $m(E_n) = n$}.
    \end{equation*}
    We call such $\varphi_X$ the fundamental sequence of X.
\end{definition}

Because $X$ is rearrangement invariant, the definition is consistent for any choice of the set $E_n$, therefore we can always put $E_n = [1,n] \cap \N$.

The fundamental sequence of BSS has some very important properties shown in \cite[Chapter 2.5]{Ben:88}, most importantly the quasi-concavity and a relation between fundamental sequences and its associate counterpart.

%\begin{proposition}\label{P:fundamental-quasi-concavity}
 %   Let $X$ be a rearrangement-invariant BSS, then the fundamental sequence satisfies:
  %  \begin{enumerate} [(i)]
   %     \item $\varphi_X(n) = 0$ iff $n = 0$,
    %    \item $\varphi_X$ is non-decreasing,
    %    \item the sequence $n \mapsto \frac{\varphi_X(n)}{n}$ is non-increasing.
   % \end{enumerate}
%\end{proposition}

%\begin{proposition} \label{P:fundamental-sequence-associate-space} \color{magenta}
 %   Let $X$ be as BSS and $X'$ its associate space, then the fundamental sequences satisfy:
  %  \begin{equation*}
   %     \varphi_X(n) \varphi_{X'}(n) = n.
   % \end{equation*}
%\end{proposition}

%The proposition comes from \cite[Theorem~I.5.2]{Ben:88}.

\medskip

\begin{proposition} \label{P:embeding-and-fundamental-sequence} 
    Let $(X,\| \cdot \|_X)$ be an r.i qBSS. Then $X\hookrightarrow \ell^\infty$ and, moreover, $X = \ell^\infty$ iff $\varphi_X \approx 1$.
\end{proposition}

\begin{myproof}
    Let $a\in X$. Then, by using Proposition~\ref{P:Properties-of-rearrangement} (ii) and the properties of $X$, we get
    \begin{equation*}
        \| a\|_\infty = \sup_{n\in \N} \abs{a_n} = a_1^*
        = a_1^* \frac{\| \chi_{\{1\}}\|_X}{\| \chi_{\{1\}}\|_X} 
        = \frac{\| a_1^* \chi_{\{1\}}\|_X}{\| \chi_{\{1\}}\|_X}
        \leq \frac{\| a^* \|_X}{\varphi_X(1)}
        = \frac{\| a \|_X}{\varphi_X(1)},
    \end{equation*}
    proving $X\hookrightarrow \ell^\infty$.

    Now suppose $\varphi_X \approx 1$, then $\varphi_X \leq C$ and thus for $a\in X$ we have
    \begin{equation*}
        \| a \|_X \leq \| \|a\|_\infty \chi_\N \|_X = \lim_{n\to \infty} \|a\|_\infty \| \chi_{[1,n]\cap \N} \|_X \leq C \|a\|_\infty.
    \end{equation*}
    Therefore $ \ell^\infty \hookrightarrow X$.
    
    If $\varphi_X \rightarrow \infty$ then:
    \begin{equation*}
        \| \chi_\N \|_X = \lim_{n\to \infty} \| \chi_{[1,n]\cap \N} \|_X = \infty,
    \end{equation*}
    but $\| \chi_\N \|_\infty = 1$. The existence of the limits comes from monotonicity of the fundamental sequence which is a direct consequence of (P2).
\end{myproof}

\section{Orlicz spaces}

The first class of sequence spaces we will work with are Orlicz spaces. The general idea is to replace the power function from the definition of Lebesgue norm with a more general Young function, giving us wider possibilities to generate more interesting spaces.

We say that a function \(A\colon [0,\infty] \rightarrow [0,\infty]\) is a Young function, if it meets the following:
\begin{enumerate}[(i)]
    \item is left continuous on \([0,\infty)\),
    \item is convex on \([0,\infty)\),
    \item $A(0)=0$,
    \item is not constant on \((0,\infty)\).
\end{enumerate}

It follows from these axioms that a Young function A is non-decreasing, $\lim_{x\to\infty} A(x)= \infty$ and it is continuous it where is finite valued. 

Every Young function could be expressed as an integral of some function $f\colon [0,\infty) \rightarrow [0,\infty]$ in the sense that
\begin{equation} \label{E:derivace_Young}
    A(t) = \int_0^t f(\tau) d\tau
    \quad \text{for $t\in [0,\infty]$.}
\end{equation}
Such function is called density function of $A$ and the function $f$ is non-decreasing and non-constant on $[0,\infty)$ and $f(0) = 0$. Moreover if $f$ satisfies those properties, then the function $A$ defined by \eqref{E:derivace_Young} is a Young function.  

For a Young function $A$ and its density function $f$ we define the complementary function $\widetilde{A}$ of the function $A$ as:
\begin{equation*}
    \widetilde{A}(t) = \int_0^t f^{-1}(\tau) d\tau \quad \text{for $t\in [0,\infty]$,}
\end{equation*}
where $f^{-1}(t) = \sup\{\tau\geq0: f(\tau)\leq t\}$ for $t\geq 0$.

In the definition of the norm using the Young function we can not just use the inverse of the integral as in the definition of the Lebesgue norm. One of the ways, to get the desired properties is to use the Minkowski functional. This gives us the following definition of the Luxemburg norm and the Orlicz space.

\begin{definition}\label{D:luxemburg-norm}
    Let $A$ be a Young function. Then we define the Luxemburg norm $\|\cdot\|_A\colon \mathcal{M} \rightarrow [0,\infty]$ as
    \begin{equation*}
        \|a\|_A = \inf\left\{ \lambda>0: \sum_{n=1}^\infty A \left( \frac{|a_n|}{\lambda}\right) \leq 1\right\}
        \quad\text{for \(a \in\mathcal{M}\).}
    \end{equation*}
    The Orlicz space \((\ell^A,\|\cdot\|_A)\) associated with $A$ is defined as
    \begin{equation*}
        \ell^A = \{a\in\mathcal{M}:\|a\|_A <\infty\}
    \end{equation*}
    with $\|\cdot\|_A$ as its norm.
\end{definition}

From the definition it follows immediately that $a\in \ell^A$ if and only if there exists $\lambda>0$ such that $\sum_{n=1}^\infty A \left( \frac{|a_n|}{\lambda}\right) < \infty$.

Many properties of Orlicz spaces are known for resonant measure spaces. Because $(m, \mathcal{P(\N)})$ is a resonant space (it is fully atomic, containing atoms of the same measure) we can use this theory. One of the results for those Orlicz spaces is that Orlicz spaces are rearrangement-invariant BSS. The proof of this and many different theorems could be found in \cite[Chapter IV.8]{Ben:88}. 

The associated space to an Orlicz space $\ell^A$ is also an Orlicz space, it is the space generated by the complementary function $\widetilde{A}$ to the Young function $A$. The proof could be found in \cite[Corollary IV.8.15]{Ben:88} 

The following observation gives us a general idea of what happens when the Young function is zero on some neighborhood of zero.

\begin{proposition}\label{P:Orlicz-for-t0}
    Let $A$ be a Young function and put $t_0 = \sup\{t\geq 0: A(t) = 0\}$.
    \begin{enumerate}
        \item [\textup{(i)}]We have $\ell^A \hookrightarrow \ell^\infty$. \label{P:Orlicz-infty-embeding}
        \item [\textup{(ii)}]If $t_0>0$ then $\ell^A = \ell^\infty$. \label{P:Orlicz-for-t0>0} 
        \item [\textup{(iii)}]If $t_0=0$ then $\ell^A \subset c_0$. \label{P:Orlicz-for-t0=0} 
    \end{enumerate} 
\end{proposition}

%\note[inline]{Luboš: Tady budeme psát poznámky a dotazy.}

\begin{myproof}
    (i) comes immediately from the fact that $\ell^A$ is a \ri BSS and Proposition \ref{P:embeding-and-fundamental-sequence}.
    
    (ii) Thanks to (i), we only need to prove $\ell^\infty \hookrightarrow \ell^A$. Suppose $\|a\|_\infty < \infty$. Then
    \begin{align*}
        \|a\|_A &= \inf\left\{ \lambda>0: \sum_{n=1}^\infty A \left( \frac{|a_n|}{\lambda}\right) \leq 1\right\} 
            \\
        &\leq \inf\left\{ \lambda>0: \sum_{n=1}^\infty A \left( \frac{\|a\|_\infty}{\lambda}\right) \leq 1\right\}
        \leq \frac{1}{t_0} \|a\|_\infty
    \end{align*}
    because $A \left( \frac{\|a\|_\infty t_0}{\|a\|_\infty}\right) = 0$, which follows from the left-continuity of $A$. This establishes the claim.

    (iii) For contradiction, assume that there is $a\in \ell^A$, $\varepsilon>0$ and $E\subset\N$ infinite such that $|a_n|\geq \varepsilon$ for every $n\in E$. For $\lambda>0$ arbitrary we have
    \begin{equation*}
        \sum_{n=1}^\infty A \left( \frac{|a_n|}{\lambda}\right) \geq \sum_{n\in E} A \left( \frac{|a_n|}{\lambda}\right) \geq \sum_{n\in E} A \left( \frac{\varepsilon}{\lambda}\right) = \infty,
    \end{equation*}
    which is a contradiction with $a\in\ell^A$.
\end{myproof}

For the Orlicz space associated with a Young function $A$ the fundamental sequence is given by the following formula:
\begin{equation*}
        \varphi_{\ell^A}(n) = \frac{1}{A^{-1}(\frac{1}{n})},
        \quad n\in\N,
\end{equation*}
where $A^{-1}$ is the right continuous inverse of $A$. The formula comes directly from the definition of the Luxemburg norm by a simple calculation, so the proof is omitted here.

It is not a hard task to characterize embeddings between Orlicz spaces, here we provide two equivalent conditions which we will use later in this paper.

\begin{proposition} \label{T:Charakterizacion-Orlicz-Embedings}
    Let $A,B$ be Young functions. Then the following statements are equivalent:
    \begin{enumerate} [(i)]
        \item [\textup{(i)}]$\ell^A \hookrightarrow \ell^B$,
        \item [\textup{(ii)}]$\varphi_{\ell^B} \ls \varphi_{\ell^A}$,
        \item [\textup{(iii)}]there exist $c>0, t_0>0$ such that $B(t)\leq A(ct)$ for every $t\in[0,t_0)$.
    \end{enumerate}
\end{proposition}

\begin{myproof}
    From (i) it immediately follows (ii).

    Suppose that (ii) holds. Let $C>0$ such that $\varphi_{\ell^B} \leq C \varphi_{\ell^A}$. Then from quasi-concavity of the fundamental sequence we have
    \begin{equation*}
        \varphi_{\ell^B}(n+1) \leq \varphi_{\ell^B}(2n) \leq 2\varphi_{\ell^B}(n) \leq 2C \varphi_{\ell^A}(n)
        \quad \text{for every $n\in\N$.}
    \end{equation*}
    Therefore from the formula for fundamental sequence we get for $A_n, B_n$ defined bellow the following
    \begin{align*}
        A_n &= \sup \left\{\tau\geq 0: A(2C\tau)\leq \frac{1}{n}\right\}
        = \frac{1}{2C}\sup \left\{\tau\geq 0: A(\tau)\leq \frac{1}{n}\right\}\\ 
        &\leq \sup \left\{\tau\geq 0: B(\tau)\leq \frac{1}{n+1}\right\} 
        = B_{n+1}.
    \end{align*}
    If there is $t_0>0$ such that $t_0\leq B_n$ for every $n\in\mathbb N$, then $B=0$ on $[0,t_0)$ so (iii) holds. If not, then by monotonicity of $B_n$ we have $B_n \searrow 0$. From the definitions of $A_n, B_n$ for $n\geq 2$ it follows that
    \begin{equation*}
        A(2Ct)>\frac{1}{n}\geq B(t) 
        \quad \text{for $t\in(A_n, B_n]$.}
    \end{equation*}
    Because $B_{n+1} \geq A_n$, we get
    \begin{equation*}
        \bigcup_{n=2}^\infty (A_n, B_n] = (0,B_2].
    \end{equation*}
    The inequality holds trivially for $t=0$, therefore
    \begin{equation*}
        A(2Ct)\geq B(t) 
        \quad \text{for $t\in [0,B_2]$.}
    \end{equation*}

    (i) follows from (iii) by a simple calculation and the definition of the Luxemburg norm.
\end{myproof}

From this Proposition it is clear that an Orlicz space is uniquely defined by its fundamental function \ie if $\varphi_{\ell^A} \approx \varphi_{\ell^B}$, then $\ell^A = \ell^B$. This property definitely does not hold generally for any (q)BSS.

The following examples of Young functions are very important for us, we will use them in the last section of this paper.

\begin{example} \label{L:Young-functions}%Let $p\in [1,\infty]$ and $\alpha \in \R$.
    \begin{enumerate} 
        \item [\textup{(a)}]Assume that either $1<p<\infty, \alpha \in \R$, or $p=1, \alpha\leq0$, then there exists a Young function $A$ satisfying \label{L:Young-p-obecne}
        \begin{equation*}
            A(t) = t^p (1-\log(t))^{\alpha p}
            \quad \text{on some neighborhood of $0$}
        \end{equation*}
        and $\|a\|_A <\infty$ iff $\sum_{n=1}^\infty \abs{a_n} < \infty$ for every $a\in \mathcal{M}$ and the complementary function $\widetilde{A}$ satisfies
        \begin{equation*}
            \widetilde{A}(t) \sim
            \begin{cases}
                t^{p'}(1-\log(t))^{-\alpha p'} &\text{if $p>1$,}
                \\
                e^{-t^\frac{1}{\alpha}} &\text{{if $p=1$.}}
            \end{cases}
        \end{equation*}
         
        \item [\textup{(b)}]Let $\alpha>0$, then there exists a Young function $A$ satisfying: \label{L:Young-p=q=infty}
        \begin{equation*}
            A(t) = \exp(-t^{-\frac{1}{\alpha}})
            \quad \text{on some neighborhood of $0$.}
        \end{equation*}
    \end{enumerate}
\end{example}

The properties of the Young functions are easy to verify by a simple calculation, so the exact proof is omitted.

\section{Lorentz--Zygmund spaces}
The Lorentz--Zygmund spaces are being governed by a weight consisting of power function from the usual Lorentz scale depending on two parameters $p,q$ and a power of a logarithm depending on parameter $\alpha$.

The following lemma talks about the asymptotic behavior of the partial sums of the weight. We will use it quite extensively in future proofs. The lemma will allow us to easily check requirements for weighted Hardy inequalities or immediately gives us the fundamental sequence of a Lorentz--Zygmund space.

The desired weighted Hardy inequalities we will use in this paper are from~\cite[Theorem 1]{Ben:91} see also~\cite[Chap. 6]{Hardy:07}.

\begin{lemma} \label{L:suma-mocnina-log}
    Let $\lambda, \tau \in \R$.
    \begin{enumerate} [(a)] 
        \item [\textup{(i)}]If $\lambda>-1$ then
        \begin{equation*}
            \sum_{k=1}^n k^\lambda (1+\log(k))^\tau \approx n^{\lambda+1} (1+\log(n))^\tau.
        \end{equation*}
        \item [\textup{(ii)}]If $\lambda=-1$ and $\tau>-1$ then
        \begin{equation*}
            \sum_{k=1}^n k^\lambda (1+\log(k))^\tau \approx (1+\log(n))^{\tau+1}.
        \end{equation*}
        \item [\textup{(iii)}]If $\lambda=-1$ and $\tau=-1$ then
        \begin{equation*}
            \sum_{k=1}^n k^\lambda (1+\log(k))^\tau \approx 1+\log(1+\log(n)).
        \end{equation*}
        \item [\textup{(iv)}]If $\lambda<-1$ then
        \begin{equation*}
            \sum_{k=n}^\infty k^\lambda (1+\log(k))^\tau \approx n^{\lambda+1} (1+\log(n))^\tau.
        \end{equation*}
        \item [\textup{(v)}]If $\lambda=-1$ and $\tau<-1$ then 
        \begin{equation*}
            \sum_{k=n}^\infty k^\lambda (1+\log(k))^\tau \approx (1+\log(n))^{\tau+1}.
        \end{equation*}
    \end{enumerate}
\end{lemma}

\begin{myproof}
    (i) Find $\varepsilon$ such $0<\varepsilon<1+\lambda$ and $n_0\in \N$ such that:
    \begin{itemize}
        \item $k^\varepsilon(1+\log(k))^\tau \geq 1$,
        \item $k^{-\varepsilon}(1+\log(k))^\tau \leq 1$,
        \item $(n+1)^{\lambda -\varepsilon+1}-1 \leq 2 n^{\lambda -\varepsilon+1}$
        \item $(n+1)^{\lambda +\varepsilon+1}-1 \geq \frac{1}{2} n^{\lambda + \varepsilon+1}$
    \end{itemize}
    for every $k\geq n_0$. Then for $n\geq n_0$ we have:
    \begin{equation} \label{E:lambda>-1 >}
        \sum_{k=1}^n k^\lambda (1+\log(k))^{\tau} 
        = \sum_{k=1}^n k^{\lambda +\varepsilon}k^{-\varepsilon}(1+\log(k))^{\tau}
        \geq n^{-\varepsilon}(1+\log(n))^{\tau} \sum_{k=1}^n k^{\lambda +\varepsilon}
    \end{equation}
    and
    \begin{equation} \label{E:lambda>-1 <}
        \sum_{k=1}^n k^\lambda (1+\log(k))^{\tau} 
        = \sum_{k=1}^n k^{\lambda -\varepsilon}k^{\varepsilon}(1+\log(k))^{\tau}
        \leq n^{\varepsilon}(1+\log(n))^{\tau} \sum_{k=1}^n k^{\lambda -\varepsilon}.
    \end{equation}

    If $\lambda+\epsilon<0$ then:
    \begin{equation} \label{E:odhad-int-1}
        \sum_{k=1}^n k^{\lambda +\varepsilon}
        \geq \int_0^n (t+1)^{\lambda +\varepsilon} dt
        = \frac{(n+1)^{\lambda +\varepsilon +1}-1}{\lambda +\varepsilon +1}
        \gs n^{\lambda +\varepsilon +1}.
    \end{equation}
    
    If $\lambda+\epsilon \geq 0$
    \begin{equation} \label{E:odhad-int-2}
        \sum_{k=1}^n k^{\lambda +\varepsilon}
        \geq \int_0^n t^{\lambda +\varepsilon} dt
        \approx  n^{\lambda +\varepsilon +1}.
    \end{equation}

    If $\lambda-\varepsilon<0$:
    \begin{equation} \label{E:odhad-int-3}
        \sum_{k=1}^n k^{\lambda -\varepsilon}
        \leq \int_0^n t^{\lambda -\varepsilon} dt
        \approx  n^{\lambda -\varepsilon +1}.
    \end{equation}
    
    If $\lambda-\varepsilon\geq 0$:
    \begin{equation} \label{E:odhad-int-4}
        \sum_{k=1}^n k^{\lambda -\varepsilon}
        \leq \int_0^n (t+1)^{\lambda -\varepsilon} dt
        = \frac{(n+1)^{\lambda -\varepsilon +1}-1}{\lambda -\varepsilon +1}
        \ls  n^{\lambda +\varepsilon +1}.
    \end{equation}
    
    Note that the transitions between the sums and the integrals are only possible thanks to the monotonicity of the summed sequence. Combining \eqref{E:lambda>-1 >} with either \eqref{E:odhad-int-1} or \eqref{E:odhad-int-2} and \eqref{E:lambda>-1 <} with either \eqref{E:odhad-int-3} or \eqref{E:odhad-int-4} gives us the desired result.

    (ii) Find $n_0 \in \N$ such that $k \mapsto k^{-1} (1+\log(k))^{\tau}$ is non-increasing for $k\geq n_0$. Put $C = \sum_{k=1}^{n_0} k^{-1}(1+\log(k))^{\tau}$, then using the substitution $x=\log(t+1)$ and the fact that $\tau>-1$ we have for $n\geq n_0$ the following:
    \begin{align*}
        \sum_{k=1}^n k^{-1}(1+\log(k))^{\tau} 
        &= C + \sum_{k=n_0+1}^n k^{-1}(1+\log(k))^{\tau} \\
        &\geq C + \int_{n_0}^n \frac{(1+\log(t+1))^{\tau}}{t+1} dt 
        = C + \int_{\log(n_0+1)}^{\log(n+1)}(1+x)^{\tau} dx \\
        &= C + \frac{(1+\log(n+1))^{\tau +1}-(1+\log(n_0+1))^{\tau +1}}{\tau +1} \\ 
        &\approx (1+\log(n+1))^{\tau +1}.
    \end{align*}
    Also by using the substitution $x=\log(t)$ we have for $n\geq n_0$ the following:
    \begin{align*}
        \sum_{k=1}^n k^{-1}(1+\log(k))^{\tau} 
        &= C + \sum_{k=n_0+1}^n k^{-1}(1+\log(k))^{\tau}
        \leq C + \int_{n_0}^n \frac{(1+\log(t))^{\tau}}{t} dt\\
        &= C + \int_{\log(n_0)}^{\log(n)}(1+x)^{\tau} dx \\
        &= C + \frac{(1+\log(n))^{\tau +1}-(1+\log(n_0))^{\tau +1}}{\tau +1}\\
        &\approx (1+\log(n+1))^{\tau +1}.
    \end{align*}

    (iii) We have
    \begin{align*}
        \sum_{k=1}^n k^{-1}(1+\log(k))^{-1} 
        &\geq \int_{0}^n \frac{1}{(1+\log(t+1))(t+1)} dt \\
        &= \int_{0}^{\log(n+1)}\frac{1}{1+x} dx = \log(1+\log(n+1)) \approx 1+\log(1+\log(n)).
    \end{align*}
    And, conversely,
    \begin{equation*}
        \sum_{k=1}^n k^{-1}(1+\log(k))^{-1} 
        \leq 1+\int_{1}^n \frac{1}{(1+\log(t))t} dt
        = 1+ \int_{0}^{\log(n)}\frac{1}{1+x} dx = 1+\log(1+\log(n)).
    \end{equation*}    

    (iv) Let $0<\varepsilon<-\lambda-1$. Find $n_0$ such that
    \begin{itemize}
        \item $n_0\geq 2$,
        \item $k^\varepsilon(1+\log(k))^\tau \geq 1$,
        \item $k^{-\varepsilon}(1+\log(k))^\tau \leq 1$,
    \end{itemize}
     for every $k\geq n_0$. Then for $n\geq n_0$ we have
     \begin{equation} \label{E:lambda<-1 >}
        \sum_{k=n}^\infty k^\lambda (1+\log(k))^{\tau} 
        = \sum_{k=n}^\infty k^{\lambda -\varepsilon}k^{\varepsilon}(1+\log(k))^{\tau}
        \geq n^{\varepsilon}(1+\log(n))^{\tau} \sum_{k=n}^\infty k^{\lambda -\varepsilon},
    \end{equation}
    \begin{equation} \label{E:lambda<-1 <}
        \sum_{k=n}^\infty k^\lambda (1+\log(k))^{\tau} 
        = \sum_{k=n}^\infty k^{\lambda +\varepsilon}k^{-\varepsilon}(1+\log(k))^{\tau}
        \leq n^{-\varepsilon}(1+\log(n))^{\tau} \sum_{k=n}^\infty k^{\lambda +\varepsilon}.
    \end{equation}
    And similarly as before
    \begin{equation} \label{E:odhad-int-5}
        \sum_{k=n}^\infty k^{\lambda -\varepsilon} 
        \geq \int_{n-1}^\infty (t+1)^{\lambda -\varepsilon} dt
        = 0 - \frac{n^{\lambda-\varepsilon+1}}{\lambda - \varepsilon +1} 
        \approx n^{\lambda-\varepsilon+1},
    \end{equation}
    \begin{equation} \label{E:odhad-int-6}
        \sum_{k=n}^\infty k^{\lambda +\varepsilon} 
        \leq \int_{n-1}^\infty t^{\lambda +\varepsilon} dt
        = 0 - \frac{(n-1)^{\lambda+\varepsilon+1}}{\lambda + \varepsilon +1} 
        \approx n^{\lambda+\varepsilon+1}.
    \end{equation}
    Therefore by combining \eqref{E:lambda<-1 >} and \eqref{E:lambda<-1 <} with \eqref{E:odhad-int-5} and \eqref{E:odhad-int-6} respectively, we get the desired result.

    (v) The assertion could be proven similarly as in (ii) and (iv).
\end{myproof}

Now we define the Lorentz--Zygmund functionals for the non-increasing rearrangement and the maximal sequence. Then we define the corresponding spaces.

\begin{definition}\label{D:lorentz-zygmund-space}
    Let $p\in(0,\infty], q\in (0,\infty], \alpha\in\mathbb{R}$ and $ a\in\mathcal{M}$, we define the functional $\|\cdot\|_{p,q;\alpha}\colon \mathcal{M}\rightarrow[0,\infty]$ as
    \begin{equation*}
        \|a\|_{p,q;\alpha} = \|a_n^* n^{\frac{1}{p}-\frac{1}{q}}(1+\log(n))^\alpha\|_q
    \end{equation*}
    and the Lorentz--Zygmund space $(\ell^{p,q;\alpha},\|\cdot\|_{p,q;\alpha})$ as
    \begin{equation*}
        \ell^{p,q;\alpha} = \{ a\in\mathcal{M}: \|a\|_{p,q;\alpha}<\infty\}.
    \end{equation*}
    We define the modified Lorentz--Zygmund functional $\|\cdot\|_{(p,q;\alpha)}\colon \mathcal{M}\rightarrow[0,\infty]$ as:
    \begin{equation*}
        \|a\|_{(p,q;\alpha)} = \| a_n^{**} n^{\frac{1}{p}-\frac{1}{q}}(1+\log(n))^\alpha \|_q
    \end{equation*}
    and the modified Lorentz--Zygmund space $(\ell^{(p,q;\alpha)},\| \cdot \|_{(p,q;\alpha)})$ as
    \begin{equation*}
        \ell^{(p,q;\alpha)}=\{a\in\mathcal{M}: \|a\|_{(p,q;\alpha)}<\infty\}.
    \end{equation*}
\end{definition}

In the following part of the section we will take a closer look on when the Lorentz--Zygmund space is a BSS. Because Lorentz--Zygmund spaces are a type of weight spaces, the biggest problem is the triangle inequality, because the operator of non-increasing rearrangement is not sub-additive.

For this reason, we have introduced the modified space $\ell^{(p,q;\alpha)}$ using the maximal sequence, which is as we know sub-additive and therefore for $q\geq 1$ the functional $\|\cdot\|_{(p,q;\alpha)}$ is a norm (for more detail see Proposition \ref{P:L-Z-BSS}). So the underlying question is whether the two functionals are equivalent. The assertion $\nrm{\cdot}_{p,q;\alpha} \ls \nrm{\cdot}_{(p,q;\alpha)}$ is obvious from definition of the maximal sequence, the converse is shown in the following proposition.

\begin{proposition}
    Let $1<p\leq\infty, 1\leq q \leq \infty, \alpha\in\R$, then $\|\cdot\|_{(p,q;\alpha)} \ls \|\cdot\|_{p,q;\alpha}$.
\end{proposition}

\begin{myproof}
    Suppose $1<q<\infty$ and $a\in\mathcal{M}$. From assumptions we have $\frac{q}{p}-q-1<-1$ and $(\frac{q}{p}-1)(1-q')>-1$ and by using Lemma \ref{L:suma-mocnina-log} (i) and (iv) we have
    \begin{align*}
        &\sup_{n\in\N}\left(\sum_{k=n}^\infty k^{\frac{q}{p}-q-1}(1+\log(k))^{q\alpha}\right)^\frac{1}{q} \left(\sum_{k=1}^n k^{(\frac{q}{p}-1)(1-q')}(1+\log(k))^{q\alpha(1-q')}\right)^\frac{1}{q'}\\
        &\leq K_1 \sup_{n\in\N}\left(n^{\frac{q}{p}-q}(1+\log(n))^{q\alpha}\right)^\frac{1}{q} \left(n^{(\frac{q}{p}-1)(1-q')+1} (1+\log(n))^{q\alpha(1-q')}\right)^\frac{1}{q'}\\
        &= K_1 \sup_{n\in\N} n^{\frac{1}{p}-1} (1+\log(n))^\alpha n^{-(\frac{q}{p}-1)\frac{1}{q}+\frac{1}{q'}} (1+\log(n))^{-\alpha}
        = K_1 <\infty.
    \end{align*}

    Now suppose $q=1$. Again we have $\frac{1}{p}-2<-1$ and thus by using Lemma \ref{L:suma-mocnina-log} (iv) we get
    \begin{align*}
        &\sup_{n\in\N}\sum_{k=n}^\infty k^{\frac{1}{p}-2} (1+\log(k))^\alpha \left(n^{\frac{1}{p}-1}(1+\log(n))^\alpha \right)^{-1} \\
        &\leq K_2 \sup_{n\in\N} n^{\frac{1}{p}-1} (1+\log(n))^\alpha n^{1-\frac{1}{p}} (1+\log(n))^{-\alpha}
        = K_2 <\infty.
    \end{align*}

    Thus for $1\leq q <\infty$ we have verified the assumptions of the weighted Hardy inequality, so we get
    \begin{equation*}
        \left(\sum_{n=1}^\infty \left(\frac{1}{n} \sum_{k=1}^n a_k^* \right)^q n^{\frac{q}{p}-1}(1+\log(n))^{q\alpha} \right)^\frac{1}{q} 
        \ls \left(\sum_{n=1}^\infty (a_n^*)^q n^{\frac{q}{p}-1}(1+\log(n))^{q\alpha}\right)^\frac{1}{q}.
    \end{equation*}

    Suppose $q=\infty$ and $a\in\mathcal{M}$ such that $\sup_{n\in\N} a_n^* n^\frac{1}{p} (1+\log(n))^\alpha = M <\infty$, otherwise the inequality is trivial. Therefore for every $n\in\N$ we have:
    \begin{equation*}
        a_n^* \leq \frac{M}{n^\frac{1}{p} (1+\log(n))^\alpha}.
    \end{equation*}
    Thus by the fact that $-\frac{1}{p}>-1$ and by using Lemma \ref{L:suma-mocnina-log} (i) we get
    \begin{align*}
        \| a_n^{**} n^\frac{1}{n} (1+\log(n))^\alpha \|_\infty
        &= \sup_{n\in\N} \frac{1}{n} \sum_{k=1}^n a_k^* n^\frac{1}{p} (1+\log(n))^\alpha \\
        &\leq M \sup_{n\in\N} \sum_{k=1}^n k^{-\frac{1}{p}} (1+\log(k))^{-\alpha} n^{\frac{1}{p}-1} (1+\log(n))^\alpha \\
        &\ls M \sup_{n\in\N} n^{-\frac{1}{p}+1} (1+\log(n))^{-\alpha} n^{\frac{1}{p}-1} (1+\log(n))^\alpha \\
        &= M,
    \end{align*}
    which concludes the proof.
\end{myproof}

\begin{proposition} \label{P:properities-L-Z}
    For parameters $p\in(0,\infty], q\in (0,\infty], \alpha\in\mathbb{R}$, the Lorentz--Zygmund space $(\ell^{p,q;\alpha},\|\cdot\|_{p,q;\alpha})$ is a rearrangement--invariant quasi-Banach sequence space.
\end{proposition}

\begin{myproof}
    Clearly, the space is rearrangement--invariant and $\|a\|_{p,q;\alpha} = 0 \iff a\equiv 0$. For $c\in \mathbb{R}$ from Proposition~\ref{P:Properties-of-rearrangement} (ii) we get $\|ca\|_{p,q;\alpha} = \abs{c}\|a\|_{p,q;\alpha}$.
    The quasi-triangular inequality is proved in the following way. From convexity, respectively concavity of a function $t \mapsto t^\lambda$ for $t\geq 0$, $\lambda>0$ we have $(t+s)^\lambda \ls t^\lambda + s^\lambda$. Let $a,b\in \measurable$, then for $q<\infty$:
    \begin{align*}
        \nrm{a+b}_{p,q;\alpha} &= \left(\sum_{n=1}^\infty ((a+b)_n^*)^q n^{\frac{q}{p}-1}(1+\log(n))^{q\alpha}\right)^\frac{1}{q} \\
        &\leq \left(\sum_{n=1}^\infty (a^*_{\floor{\frac{n+1}{2}}} + b^*_{\floor{\frac{n+1}{2}}})^q n^{\frac{q}{p}-1}(1+\log(n))^{q\alpha}\right)^\frac{1}{q} \\
        &\ls \left(\sum_{n=1}^\infty (a^*_{\floor{\frac{n+1}{2}}})^q n^{\frac{q}{p}-1}(1+\log(n))^{q\alpha}\right)^\frac{1}{q} + \left(\sum_{n=1}^\infty (b^*_{\floor{\frac{n+1}{2}}})^q n^{\frac{q}{p}-1}(1+\log(n))^{q\alpha}\right)^\frac{1}{q}\\
        &= S_1 + S_2.
    \end{align*}
    If we sum over odd and even numbers separately:
    \begin{align*}
        S_1&= \left(\sum_{k=1}^\infty (a^*_k)^q (2k-1)^{\frac{q}{p}-1}(1+\log(2k-1))^{q\alpha} + \sum_{k=1}^\infty (a^*_k)^q (2k)^{\frac{q}{p}-1}(1+\log(2k))^{q\alpha}\right)^\frac{1}{q}\\
        &\approx \left(\sum_{k=1}^\infty (a^*_k)^q (k)^{\frac{q}{p}-1}(1+\log(k))^{q\alpha} + \sum_{k=1}^\infty (a^*_k)^q (k)^{\frac{q}{p}-1}(1+\log(k))^{q\alpha}\right)
        \approx \nrm{a}_{p,q;\alpha}.
    \end{align*}
    The sum $S_2$ is calculated similarly. Therefore we obtain $\nrm{a+b}_{p,q;\alpha} \ls \nrm{a}_{p,q;\alpha} + \nrm{b}_{p,q;\alpha}$. Now consider the case $q=\infty$:
    \begin{align*}
        \nrm{a+b}_{p,\infty;\alpha} &= \sup_{n\in\N} (a + b)_n^* n^{\frac{1}{p}} (1+\log(n))^\alpha\\
        &= \sup (\{(a + b)_{2k-1}^* (2k-1)^{\frac{1}{p}} (1+\log(2k-1))^\alpha; k\in\N \} \cup\\
        &\{(a + b)_{2k}^* (2k)^{\frac{1}{p}} (1+\log(2k))^\alpha; k\in\N \}) \\
        &\leq \sup (\{(a_k^* + b_k^*) (2k-1)^{\frac{1}{p}} (1+\log(2k-1))^\alpha; k\in\N \} \cup\\
        &\{(a_{k}^* + b_{k}^*) (2k)^{\frac{1}{p}} (1+\log(2k))^\alpha; k\in\N \}) \\
        &\ls \sup_{k\in\N} (a_{k}^* + b_{k}^*) (k)^{\frac{1}{p}} (1+\log(k))^\alpha \\
        &\leq \sup_{k\in\N} a_{k}^* (k)^{\frac{1}{p}} (1+\log(k))^\alpha + \sup_{k\in\N} b_{k}^* (k)^{\frac{1}{p}} (1+\log(k))^\alpha = \nrm{a}_{p,\infty;\alpha} + \nrm{b}_{p,\infty;\alpha}.
    \end{align*}
    The rest of (Q1) follows from the definition of the rearrangement. (P\ref{D:BSS-monotonie}) comes from
    Proposition~\ref{P:Properties-of-rearrangement} (ii) and the properties of $q$ (quasi-)norm. (P\ref{D:BSS-konvergence}) follows immediately from Proposition~\ref{P:Properties-of-rearrangement} (iv) and the monotone convergence theorem. For $E\subset\N, m(E)<\infty$ we have:
    \begin{equation*}
        \|\chi_E\|_{p,q;\alpha} = \left(\sum_{n=1}^{m(E)} n^{\frac{q}{p}-1}(1+\log(n))^{q\alpha}\right)^\frac{1}{q} <\infty
        \quad \text{if $q<\infty$}
    \end{equation*}
    and for $q=\infty$ we get:
    \begin{equation*}
        \|\chi_E\|_{p,q;\alpha} = \sup_{n\leq m(E)} n^\frac{1}{p}(1+\log(n))^{\alpha} <\infty.
    \end{equation*}
    This gives us (P\ref{D:BSS-char. posloupnost}). (P\ref{D:BSS-lok.l1}) holds, because (P\ref{D:BSS-char. posloupnost}), (P\ref{D:BSS-monotonie}), positive homogeneity and rearrangement--invariance.
\end{myproof}

\begin{proposition} \label{P:L-Z-BSS}
    Let $p\in (1,\infty], q\in [1,\infty], \alpha\in\R$ or $p=1, q\in [1,\infty], \alpha<-\frac{1}{q}$. Then $\ell^{(p,q;\alpha)}$ is a rearrangement--invariant BSS.
\end{proposition}

\begin{myproof}
    Similarly as in the proof of Proposition \ref{P:properities-L-Z} we have $\|a\|_{(p,q;\alpha)} = 0 \iff a\equiv 0$, $\|ca\|_{(p,q;\alpha)} = \abs{c}\|a\|_{(p,q;\alpha)}$ and the triangular inequality follows from the sub-additivity of maximal sequence and the Minkowski inequality. Axioms (P\ref{D:BSS-abs})-(P\ref{D:BSS-konvergence}) hold from same arguments as in the proof of Proposition \ref{P:properities-L-Z}. To prove (P\ref{D:BSS-char. posloupnost}) let $E\subset \N$ be such that $m(E)<\infty$ and $q<\infty$, then we get
    \begin{align*}
        \| \chi_E \|^q_{(p,q;\alpha)} 
        &= \sum_{n=1}^\infty \left(\frac{1}{n} \sum_{k=1}^n \chi_{[1,m(E)]}(k) \right)^q n^{\frac{q}{p}-1}(1+\log(n))^{q\alpha}\\ 
        &\leq m(E)^q \sum_{n=1}^\infty n^{\frac{q}{p}-q-1} (1+\log(n))^{q\alpha} 
        < \infty,
    \end{align*}
    because either $\frac{q}{p}-q-1<-1$ or $\frac{q}{p}-q-1=-1$ and $q\alpha<-1$. If $q=\infty$ we have:
    \begin{align*}
        \| \chi_E \|^q_{(p,q;\alpha)} 
        &= \sup_{n\in\N} \frac{1}{n} \sum_{k=1}^n \chi_{[1,m(E)]}(k) n^\frac{1}{p} (1+\log(n))^\alpha \\ 
        &\leq m(E) \sup_{n\in\N} n^{\frac{1}{p}-1} (1+\log(n))^\alpha <\infty,
    \end{align*}
    because either $\frac{1}{p}-1<0$ or $\frac{1}{p}-1=0$ and $\alpha<0$.
    
    The axiom (P\ref{D:BSS-lok.l1}) is satisfied by the same argument as in Proposition \ref{P:properities-L-Z}. 
\end{myproof}

If $q=1$, $p\geq 1$ and $\alpha\leq0$, then the weight sequence is non-increasing, therefore by a result from \cite[Theorem 1]{Lorentz:51} and \cite[pg.~1932, line 7]{Ci-Le:19} we have $\| \cdot \|_{p,q;\alpha}$ is a norm, which with Proposition \ref{P:properities-L-Z} gives us that the space $\ell^{p,q;\alpha}$ is a r.i.~BSS.
Therefore, for parameters either $p,q \in (1,\infty], \alpha\in\R$ or $q=1, p\in [1,\infty], \alpha \leq0$ the space $\ell^{p,q;\alpha}$ either is a r.i.~BSS or coincides with one.

If $p,q,\alpha$ does not satisfy the conditions above, we still know, that $\ell^{p,q;\alpha}$ is a \ri q-BSS, fortunately in this paper, we will not work with those edge cases.

Important question for our application is how the fundamental sequence of Lorentz--Zygmund spaces behave near infinity. This is fully answered in the following proposition.

\begin{proposition} \label{P:fundmental-sequence-L-Z}
    Let $p,q \in (0,\infty], \alpha \in \R$.
    \begin{enumerate} [(i)]
        \item If $p<\infty$, then $\varphi_{\ell^{p,q;\alpha}} \approx n^\frac{1}{p}(1+\log(n))^\alpha$.
        
        \item If $p=\infty, q<\infty, \alpha > -\frac{1}{q}$, then $\varphi_{\ell^{p,q;\alpha}} \approx (1+\log(n))^{\alpha+\frac{1}{q}}$.
        
        \item If $p=\infty, q<\infty, \alpha = -\frac{1}{q}$, then $\varphi_{\ell^{p,q;\alpha}} \approx (1+ \log(1+\log(n)))^{-\alpha}$.

        \item If $p=\infty, q<\infty, \alpha < -\frac{1}{q}$, then $\varphi_{\ell^{p,q;\alpha}} \approx 1$.
        
        \item If $p=\infty, q=\infty, \alpha >0$, then $\varphi_{\ell^{p,q;\alpha}} \approx (1+\log(n))^\alpha$.

        \item If $p=\infty, q=\infty, \alpha \leq0$, then $\varphi_{\ell^{p,q;\alpha}} \approx 1$.
    \end{enumerate}
\end{proposition}

\begin{myproof}
    (i) If $q<\infty$, then the proposition is a direct consequence of Lemma \ref{L:suma-mocnina-log} (i).
    
    Let $q=\infty$. Then there is $n_0$ such that for every $n\geq n_0$ and $k\leq n$ we have
    \begin{equation*}
        k^\frac{1}{p}(1+\log(k))^\alpha \leq n^\frac{1}{p}(1+\log(n))^\alpha.
    \end{equation*}
    Then for every such $n\geq n_0$ we have:
    \begin{equation*}
        \|\chi_{E_n}\|_{p,\infty,\alpha} = \sup_{k\leq n} k^\frac{1}{p}(1+\log(k))^\alpha = n^\frac{1}{p}(1+\log(n))^\alpha. 
    \end{equation*}
    This immediately gives us the claim.

    (ii) The claim comes again immediately form Lemma \ref{L:suma-mocnina-log} (ii).
    
    (iii) This is again a direct consequence of Lemma \ref{L:suma-mocnina-log} (iii)

    (iv) From the fact that $q\alpha <-1$ we get that the sum $\sum_{k=1}^\infty k^{-1}(1+\log(k))^{q\alpha}$ converges, therefore we have:
    \begin{equation*}
        1 \leq \varphi_{\ell^{\infty,q;\alpha}} \leq (\sum_{k=1}^\infty k^{-1}(1+\log(k))^{q\alpha})^\frac{1}{q}.
    \end{equation*}

    (v), (vi) For $p=q=\infty$ we have
    \begin{equation*}
        \varphi_{\ell^{\infty,\infty;\alpha}}(n) = \sup_{k\leq n} (1+\log(k))^\alpha = 
        \begin{cases}
            (1+\log(n))^\alpha
            &\text{for $\alpha>0$}
                \\
            1
            &\text{for $\alpha\leq 0$}
        \end{cases}
    \end{equation*}
    which immediately gives us (v) and (vi).
\end{myproof}

The next proposition partially answers the question about the form of the associate space for Lorentz--Zygmund space. For our applications we do not need the edge cases.

\begin{proposition} \label{P:Dualita-L-Z}
    Let $p,q,\alpha$ satisfy one of the following:
    \begin{enumerate}
        \item [\textup{(i)}]$p \in [1,\infty),q\in (1,\infty), \alpha\in\R$,
        \item [\textup{(ii)}]$p \in [1,\infty), q=1, \alpha \leq0$,
        \item [\textup{(iii)}]$p \in (1,\infty],q=\infty, \alpha\geq0$.
    \end{enumerate}
    Then the associate space $(\ell^{p,q;\alpha})'$ coincides with $\ell^{p',q';-\alpha}$.
\end{proposition}

\begin{myproof}
    The embedding $\ell^{p',q';-\alpha} \hookrightarrow (\ell^{p,q;\alpha})'$ is proven in the same way, regardless of conditions (i)-(iii). Let $a,b\in\mathcal{M}$, then by using the Hölder inequality and the discrete version of the Hardy--Littlewood inequality, see e.g.~\cite{Har:88,Ben:88}, we get
    \begin{align*}
        \sum_{n=1}^\infty a_n b_n 
        &\leq \sum_{n=1}^\infty a_n^* b_n^*
        = \sum_{n=1}^\infty a_n^* n^{\frac{1}{q}-\frac{1}{p}} (1+\log(n))^{-\alpha} b_n^* n^{\frac{1}{p}-\frac{1}{q}}(1+\log(n))^\alpha \\
        & \leq \| a_n^* n^{\frac{1}{q}-\frac{1}{p}} (1+\log(n))^{-\alpha}\|_{q'} \| b_n^* n^{\frac{1}{p}-\frac{1}{q}}(1+\log(n))^\alpha \|_{q}
        = \| a\|_{p',q';-\alpha} \| b\|_{p,q;\alpha}.
    \end{align*}
    Thus by taking the supremum over $\| b\|_{p,q;\alpha} \leq 1$ we get
    \begin{equation*}
        \|a\|_{(\ell^{p,q;\alpha})'} = \sup\left\{\sum_{n=1}^\infty a_n b_n :\| b\|_{p,q;\alpha} \leq 1\right\} \leq \| a\|_{p',q';-\alpha}.
    \end{equation*}
    This gives us $\ell^{p',q';-\alpha} \hookrightarrow (\ell^{p,q;\alpha})'$.

    Suppose (i) and let $a\in\mathcal{M}$, set
    \begin{equation*}
        \varrho(n) = (a^{**}_n)^{q'-1} n^{\frac{q'}{p'}-1}(1+\log(n))^{-\alpha q'}
    \end{equation*}
    and
    \begin{equation*}
        g(n) = \sum_{k=n}^\infty \frac{\varrho(k)}{k}.
    \end{equation*}
    From this and Proposition \ref{P:rearrangement-of-non-increasing} (i) we have $g^* = g$. Using the Fubini Theorem and the Hölder inequality for associate spaces we get
    \begin{align*}
        \| a \|^{q'}_{p',q' ;-\alpha} 
        &\leq \sum_{n=1}^\infty (a^{**}_n)^{q'} n^{\frac{q'}{p'}-1}(1+\log(n))^{-\alpha q'}
        =\sum_{n=1}^\infty a^{**}_n \varrho(n) 
        = \sum_{n=1}^\infty \frac{1}{n} \sum_{k=1}^n a_k^* \varrho(n) \\
        &= \sum_{k=1}^\infty \sum_{n=k}^\infty \frac{1}{n} a_k^* \varrho(n)
        = \sum_{k=1}^\infty a_k^* g(k) 
        \leq \|g\|_{p,q;\alpha} \|a\|_{(\ell^{p,q;\alpha})'}.
    \end{align*}
    To prove $(\ell^{p,q;\alpha})' \hookrightarrow \ell^{p',q';-\alpha}$ it is enough to show that $\|g\|_{p,q;\alpha} \ls \|a\|_{\ell^{p',q';-\alpha}}^{q'-1}$. \\
    Set $b_n = (a^{**}_n)^{q'-1}n^{\frac{q'}{p'}-2}(1+\log(n))^{-q'\alpha}$, $u_n = n^{\frac{q}{p}-1}(1+\log(n))^{q\alpha}$, then by using the weighted Hardy inequality we have
    \begin{align} \label{E:L-Z-associate-Hardy}
        \|g\|_{p,q;\alpha} 
        &= \left(\sum_{n=1}^\infty \left( \sum_{k=n}^\infty (a_k^{**})^{q'-1}k^{\frac{q'}{p'}-2} (1+\log(k))^{-q'\alpha}\right)^q n^{\frac{q}{p}-1}(1+\log(n))^{q\alpha} \right)^\frac{1}{q} \\
        & = \left(\sum_{n=1}^\infty \left( \sum_{k=n}^\infty b_k \right)^q u_n \right)^\frac{1}{q} \nonumber
        \leq q \left(\sum_{n=1}^\infty u_n^{1-q} \left( \sum_{k=1}^n u_k\right)^q b_n^q \right)^\frac{1}{q} \\
        &\leq q \left(\sum_{n=1}^\infty \left(n^{\frac{1}{p}-\frac{1}{q}}(1+\log(n))^\alpha\right)^{q(1-q)} \left( \sum_{k=1}^n n^{\frac{q}{p}-1}(1+\log(n))^{q\alpha}\right)^q b_n^q \right)^\frac{1}{q}. \nonumber
    \end{align}
    Because $\frac{q}{p}-1>-1$, then from Lemma \ref{L:suma-mocnina-log} (i) and \eqref{E:L-Z-associate-Hardy} we get
    \begin{align} \label{E:L-Z-associate-lemma} 
        \|g\|_{p,q;\alpha}
        &\ls \left(\sum_{n=1}^\infty \left( n^{(\frac{1}{p}-\frac{1}{q})(1-q) +\frac{q}{p}} (1+\log(n))^{\alpha (1-q) +q\alpha} b_n \right)^q \right)^\frac{1}{q}\\
        &=  \left(\sum_{n=1}^\infty \left( n^{\frac{1}{p}+\frac{1}{q'}} (1+\log(n))^{\alpha} b_n \right)^q \right)^\frac{1}{q}
        =  \| n^{\frac{1}{p}+\frac{1}{q'}} (1+\log(n))^{\alpha} b_n \|_q. \nonumber
    \end{align}
    But because $\frac{q'}{q}= q'-1$ we have
    \begin{align*}
        \|a\|^{q'-1}_{p',q';-\alpha}
        &= \left(\sum_{n=1}^\infty (a_n^{**})^{q'} n^{\frac{q'}{p'}-1}(1+\log(n))^{-\alpha q'}\right)^{\frac{q' -1}{q'}}\\ 
        &= \left(\sum_{n=1}^\infty (a_n^{**})^{q (q'-1)} n^{q\frac{q'-1}{p'}-1}(1+\log(n))^{-q\alpha (q'-1)}\right)^{\frac{1}{q}} \\
        &= \| (a_n^{**})^{q'-1} n^{\frac{q'-1}{p'}-\frac{1}{q}} (1+\log(n))^{-\alpha (q'-1)} \|_q
        = \| n^{\frac{1}{p}+\frac{1}{q'}} (1+\log(n))^{\alpha} b_n \|_q, 
    \end{align*}
    giving us the desired result.

    Now suppose (ii). From the form of fundamental sequence for associate space and Proposition \ref{P:fundmental-sequence-L-Z} we have
    \begin{equation*}
        \varphi_{(\ell^{p,1;\alpha})'} (n)
        = \frac{n}{\varphi_{\ell^{p,1;\alpha}}(n)} 
        \approx n^{1-\frac{1}{p}} (1+\log(n))^{-\alpha}
        = n^{\frac{1}{p'}} (1+\log(n))^{-\alpha}.
    \end{equation*}
    Then by using the Hölder inequality for associate spaces we have
    \begin{equation*}
        \sum_{k=1}^n a_n^* \leq \| \chi_{[1,n]} \|_{p,1;\alpha} \| a \|_{(\ell^{p,1;\alpha})'}
        = \varphi_{\ell^{p,q;\alpha}}(n) \| a \|_{(\ell^{p,1;\alpha})'}
        = \frac{n}{\varphi_{(\ell^{p,q;\alpha})'}(n)} \| a \|_{(\ell^{p,1;\alpha})'} 
    \end{equation*}
    for $a\in\mathcal{M}, n\in\N$.
    Thus we get
    \begin{equation*}
        \|a\|_{p',\infty;-\alpha} \leq \sup_{n\in\N} a_n^{**} n^\frac{1}{p'} (1+\log(n))^{-\alpha}
        \ls \sup_{n\in\N} a_n^{**} \varphi_{(\ell^{p,1;\alpha})'}(n)
        \leq \| a \|_{(\ell^{p,1;\alpha})'}.
    \end{equation*}

    Suppose (iii), because $\ell^{p',q';-\alpha}$ is a BSS we have from \cite[Chapter 1, Theorem 2.7]{Ben:88} and the case~(ii)
    \begin{equation*}
        (\ell^{p,q;\alpha})' = ((\ell^{p',q';-\alpha})')' = (\ell^{p',q';-\alpha})'' = \ell^{p',q';-\alpha}.
    \end{equation*}
    This concludes the proof.    
\end{myproof}

\section{Relations between Orlicz and Lorentz--Zygmund spaces}

In this chapter, we will characterize when a Lorentz--Zygmund space is an Orlicz space. More specifically, for which parameters $p,q, \alpha$ there exists a Young function $A$ such that the norms $\| \cdot \|_{p,q;\alpha}$ and $\| \cdot \|_A$ are equivalent.

To achieve this, we will need to formulate a few lemmas. The scale of L--Z spaces is sharp in the sense, that for different parameters $p,q,\alpha$ we almost always get a different space.

\begin{lemma} \label{L:L-Z-difference}
    For $p,r,q,s \in (0,\infty], \alpha,\beta \in \R$, we have $\ell^{p,q;\alpha} = \ell^{r,s;\beta}$ if and only if one of the following conditions holds:
    \begin{enumerate}
        \item [\textup{(1)}]$p=r, q=s, \alpha = \beta$,
        \item [\textup{(2)}]$(p,q;\alpha)$ and $(r,s;\beta)$ each satisfy one of the conditions (iv) or (vi) from Proposition~\ref{P:fundmental-sequence-L-Z}. 
    \end{enumerate}
    Furthermore, if \textup{(2)} is true, then $\ell^{p,q;\alpha} = \ell^{r,s;\beta} = \ell^\infty$.
\end{lemma}

\begin{myproof}
    If (1) holds, the proposition is trivial. If (2) is true then from Proposition~\ref{P:fundmental-sequence-L-Z} we get
    \begin{equation*}
        \varphi_{\ell^{p,q;\alpha}} \approx \varphi_{\ell^{r,s;\beta}} \approx 1. 
    \end{equation*}
    Then using Proposition~\ref{P:embeding-and-fundamental-sequence} we have $\ell^{p,q;\alpha} = \ell^\infty = \ell^{r,s;\beta}$.

    Let $\ell^{p,q;\alpha} = \ell^{r,s;\beta}$ and for contradiction suppose that both (1) and (2) fail. Then $\varphi_{\ell^{p,q;\alpha}} \approx \varphi_{\ell^{r,s;\beta}}$. From Proposition \ref{P:fundmental-sequence-L-Z} we get two possible cases:
    \begin{enumerate} [(a)]
        \item $p=r, \alpha = \beta$ and $q\neq s$,
        \item $p=r=\infty, \alpha + \frac{1}{q} = \beta + \frac{1}{s}>0$ and $q \neq s$. 
    \end{enumerate}
    Without loss of generality suppose $q>s$.

    If (a) holds, find $u\in \R$ such that $-\frac{1}{s}-\alpha < u <-\frac{1}{q} -\alpha$ and define
    \begin{equation*}
        a_n = n^{-\frac{1}{p}}(1+\log(n))^u.
    \end{equation*}
    From the choice of $u$ we get that $a$ is non-increasing, therefore by Proposition \ref{P:rearrangement-of-non-increasing} (i) $a=a^*$. Then from the definition $a\in \ell^{p,q;\alpha}$ but $a\notin \ell^{r,s;\beta}$, thus giving us the contradiction.
    
    If (b) holds, find $v \in \R$ such that $-\frac{1}{s}< v <-\frac{1}{q}$ and for $n\in \N$ set
    \begin{equation*}
        a_n = (1+\log(n))^{-(\alpha+\frac{1}{q})} (1+\log(1+\log(n)))^v.
    \end{equation*}
    Because $v<0$, $a$ is non-increasing. Therefore again by Proposition \ref{P:rearrangement-of-non-increasing} (i) we have $a=a^*$. Then again $a\in \ell^{p,q;\alpha}\setminus \ell^{r,s;\beta}$.
\end{myproof}

Now recall Example \ref{L:Young-functions}, we need to show that those functions have the desired properties, \ie that they create the same spaces as certain Lorentz--Zygmund spaces.

In the very technical proof of Lemma \ref{L:odhad-p<infty} we will use a certain version of Young inequality. To avoid confusion, we will formulate it here explicitly. The Young inequality could be found in \cite[Lemma 8.3]{Opick:99}.

\begin{lemma} \label{L:Young-inequality}
    For $a,b>1$ and $\lambda >0$ we have
    \begin{equation*}
        ab\leq \exp(a^\frac{1}{\lambda}) + b (1+\log(b))^\lambda.
    \end{equation*}
\end{lemma}

\begin{lemma} \label{L:odhad-p<infty}
    Let $1\leq p <\infty$ and $\lambda\ge0$ . Then for every $a\in \mathcal{M}$, we have
    \begin{equation} \label{E:odhad-p<infty-n}
        \sum_{n=1}^\infty (a_n^*)^p (1+\log(n))^\lambda < \infty
    \end{equation}
    if and only if
    \begin{equation} \label{E:odhad-p<infty-a_n}
        \sum_{n=1}^\infty (a_n^*)^p (1-\log(a_n^*))^\lambda < \infty.
    \end{equation}
\end{lemma}

Because $\lim_{t\to 0_+} t^p (1-\log(t))^\lambda = 0$, \eqref{E:odhad-p<infty-a_n} is consistent.

\begin{proof}[Proof of Lemma \ref{L:odhad-p<infty}]
    Let $a\in\mathcal{M}$ satisfying either \eqref{E:odhad-p<infty-n} or \eqref{E:odhad-p<infty-a_n}, then $a^*_n \to 0$, therefore we can without loss of generality assume that $a^* < 1$ and $a_n^* \neq 0$.

    If $\lambda=0$, the claim trivially holds.

    Now suppose $\lambda>0$ and $a$ satisfy \eqref{E:odhad-p<infty-a_n}. First we need to prove that 
    \begin{equation*}
        \sup_{n\in\N} n^\frac{1}{p}a^*_n = K <\infty.
    \end{equation*}
    For contradiction suppose that there exists an increasing sequence $\{n_k\}_{n=1}^\infty$ of natural numbers such that $n_k^\frac{1}{p} a_{n_k} \to \infty$. For every $k\in \N$ we have
    \begin{align*} 
        \sum_{n=1}^\infty (a_n^*)^p (1-\log(a_n^*))^\lambda & \nonumber
        \geq \sum_{n=\floor{\frac{n_k}{2}}}^{n_k} (a_n^*)^p (1-\log(a_n^*))^\lambda \\ 
        &\geq (a_{n_k}^*)^p (1-\log(a_{\floor{\frac{n_k}{2}}}^*))^\lambda \sum_{n=\floor{\frac{n_k}{2}}}^{n_k} 1 \\ \nonumber
        &\geq \frac{1}{2} n_k (a_{n_k}^*)^p (1-\log(a_1^*))^\lambda \to \infty
        \quad \text{for $k\to\infty$,}
    \end{align*}
    giving us the contradiction. Therefore for every $n\in\N$ we have $\frac{n^\frac{1}{p}}{K}\leq \frac{1}{a_n^*}$. It is easy to see that \eqref{E:odhad-p<infty-n} is equivalent to
    \begin{equation*}
        \sum_{n=1}^\infty (a_n^*)^p \left(1+\log\left(\frac{n^\frac{1}{p}}{K}\right)\right)^\lambda < \infty
    \end{equation*}
    but we have
    \begin{equation*}
        \sum_{n=1}^\infty (a_n^*)^p \left(1+\log\left(\frac{n^\frac{1}{p}}{K}\right)\right)^\lambda 
        \leq \sum_{n=1}^\infty (a_n^*)^p (1-\log(a_n^*))^\lambda < \infty
    \end{equation*}
    as desired.

    Now suppose that for $a\in\mathcal{M}$ \eqref{E:odhad-p<infty-n} holds. From the assumption above we have $a^* <1$, therefore there is $k_0\in \N_0$ such that
    \begin{equation*}
        e^{\frac{-(k_0+1)}{p}}< a^*_1 \leq e^\frac{-k_0}{p}.
    \end{equation*}
    For every $k\geq k_0$ set $n_k = \min\{n\in\N: a_n^* \leq e^\frac{-k}{p}\}$. Because $a_n^*>0$, then $\{n_k\}$ is unbounded, therefore $n_k \nearrow \infty$. From the definition of $n_k$ we have
    \begin{equation} \label{E:approx-a^p}
        \frac{1}{e} e^{-k} < (a_{n}^*)^p \leq e^{-k}
    \end{equation}
    \begin{equation} \label{E:approx-log(a)}
        \frac{1}{p^\lambda} k^\lambda < (1-\log(a_{n}^*))^\lambda \leq \left(\frac{3}{p}\right)^\lambda k^\lambda
    \end{equation}
    for every $n\in \{n_k, \dots , n_{k+1}-1\}$. Now using \eqref{E:approx-a^p} and \eqref{E:approx-log(a)} we get
    \begin{align*}
        \sum_{n=1}^\infty (a_n^*)^p (1-\log(a_n^*))^\lambda &
        =\sum_{k=k_0}^\infty \sum_{n=n_k}^{n_{k+1}-1} (a_n^*)^p (1-\log(a_n^*))^\lambda \\
        &\ls \sum_{k=k_0}^\infty e^{-k} k^\lambda (n_{k+1} - n_k) 
        =  \sum_{k=k_0}^\infty e^{-k} a_k b_k,
    \end{align*}
    where $a_k = (\frac{k}{2})^\lambda, b_k = 2^\lambda(n_{k+1} - n_k)$. Using Lemma \ref{L:Young-inequality} on $a_k,b_k$ we get
    \begin{align*}
        \sum_{k=k_0}^\infty e^{-k} a_k b_k &
        \leq \sum_{k=k_0}^\infty e^{\frac{-k}{2}} + \sum_{k=k_0}^\infty e^{-k} 2^\lambda (n_{k+1}-n_k)(1+\log(2^\lambda (n_{k+1}-n_k)) \\
        &= S_1 + S_2,
    \end{align*}
    where $S_1<\infty$ and
    \begin{equation*}
        S_2 < \infty \iff \sum_{k=k_0}^\infty e^{-k} (n_{k+1} -n_k) (1+\log(n_{k+1} -n_k))^\lambda <\infty.
    \end{equation*}
    For every $k\geq k_0$ we have
    \begin{align} \label{E:n_k-estimate}
        &(n_{k+1} -n_k) (1+\log(n_{k+1} -n_k))^\lambda \nonumber
        \leq (n_{k+1} -n_k) (1+\log(n_{k+1}))^\lambda \\ 
        &\leq  n_{k+1} (1+\log(n_{k+1}))^\lambda - n_k (1+\log(n_{k}))^\lambda.
    \end{align}
    Find $k_1\geq k_0$ such that $\lambda\leq 1+\log(n_{k_1})$. Thus by using \eqref{E:n_k-estimate} and the fact that
    \begin{equation*}
        \frac{d}{dt}\left(t(1+\log(t))^\lambda\right)= (1+\log(t))^{\lambda-1}(1+\log(t) +\lambda),
    \end{equation*}
    we get
    \begin{align*}
        &\sum_{k=k_0}^\infty e^{-k} (n_{k+1} -n_k) (1+\log(n_{k+1} -n_k))^\lambda \\
        &\leq \sum_{k=k_0}^\infty e^{-k} n_{k+1} (1+\log(n_{k+1}))^\lambda - n_k (1+\log(n_{k}))^\lambda\\
        &=  \sum_{k=k_0}^\infty e^{-k} \int_{n_k}^{n_{k+1}} (1+\log(t))^{\lambda-1}(1+\log(t) +\lambda) \\
        &\ls C + \sum_{k=k_1}^\infty e^{-k} \int_{n_k}^{n_{k+1}} (1+\log(t))^{\lambda} 
        \leq C + \sum_{k=k_1}^\infty e^{-k} \sum_{n=n_k}^{n_{k+1}-1} (1+\log(n))^{\lambda}\\
        &\ls C +\sum_{k=k_1}^\infty \sum_{n=n_k}^{n_{k+1}-1} (a_n^*)^p (1+\log(n))^{\lambda}
        = C + \sum_{n=n_{k_1}}^\infty (a_n^*)^p (1+\log(n))^{\lambda} <\infty,
    \end{align*}
    because $C$ is a finite sum, which is always less than infinity. Thus we get what we desired.
    \newline
\end{proof}

\begin{lemma} \label{L:Odhad-p=infty}
    Let $\alpha>0$, $A$ Young function such that $A (t)\sim \exp(-t^{-\frac{1}{\alpha}})$ and $a\in\mathcal{M}$. Then $\|a\|_{\infty,\infty;\alpha}<\infty$ if and only if $\|a\|_A <\infty$.
\end{lemma}

\begin{myproof}
    Let $a \in \mathcal{M}$ be such that
    \begin{equation*}
        \|a\|_{\infty,\infty;\alpha} = \sup_{n\in \N} a^*_n (1+\log(n))^\alpha = K <\infty.
    \end{equation*}
    Then for every $n\in\N$ we have $a_n^* \leq \frac{K}{(1+\log(n))^\alpha}$. Let $c,C>0$ such that $\exp\left(-(ct)^{-\frac{1}{\alpha}}\right) \leq A(t) \leq \exp\left(-(Ct)^{-\frac{1}{\alpha}}\right)$ on some positive neighborhood of $0$ $U$. Find $\lambda >2KC$ such that $\frac{a^*_n}{\lambda}\in U$ for every $n\in\N$. From the fact that $\ell^A$ is \ri we have
    \begin{align*}
        \sum_{n=1}^\infty A\left(\left|{\frac{a_n}{\lambda}}\right|\right) &\leq \sum_{n=1}^\infty \exp\left(-\left(\frac{C a_n^*}{\lambda}\right)^{-\frac{1}{\alpha}}\right)
        \leq \sum_{n=1}^\infty \exp\left(-\left(2^{\frac{1}{\alpha}}(1+\log(n))\right)\right) \\
        &\approx \sum_{n=1}^\infty \exp \left(\log\left(n^{-2^\frac{1}{\alpha}}\right)\right)
        = \sum_{n=1}^\infty n^{-2^\frac{1}{\alpha}} < \infty
    \end{align*}
    because $-2^\frac{1}{\alpha}<-1$. Therefore $\|a\|_A <\infty$.

    Now suppose $a\in \ell^A$. Then there is $\lambda >0$ such that $\frac{|a_n|}{\lambda}\in U$ and $\sum_{n=1}^\infty A(\abs{\frac{a_n}{\lambda}}) \leq 1$. For contradiction suppose that there exists an increasing sequence $\{ n_k\}_{k=1}^\infty$ of natural numbers such that $a^*_{n_k}(1+\log(n_k))^\alpha \to \infty$ for $k\to \infty$. This gives us:
    \begin{align*}
        1 &\geq \sum_{n=1}^\infty A\left(\frac{a^*_n}{\lambda}\right) 
        \geq \sum_{n=\floor{\frac{n_k}{2}}}^{n_k} \exp\left(-\left(\frac{ca_n^*}{\lambda}\right)^{-\frac{1}{\alpha}}\right)
        \\
        &\gs n_k \exp\left(-\left(\frac{ca_{n_k}^* (1+\log(n_k))^\alpha}{\lambda 2^\alpha}\right)^{-\frac{1}{\alpha}} \frac{1+\log(n_k)}{2}\right) \\
        &=n_k^\frac{1}{2} \exp\left(-\left(\frac{ca_{n_k}^* (1+\log(n_k))^\alpha}{\lambda 2^\alpha}\right)^{-\frac{1}{\alpha}} \right) 
        \to \infty
        \quad \text{as $k\to \infty$.}
    \end{align*}
    This gives us the contradiction.
\end{myproof}

Now we can formulate and prove the final result.

\begin{theorem} \label{T: Charakterization-of-intersection}
    Let $p,q\in (0,\infty], \alpha \in \R$ and $A$ a Young function. Then $\ell^{p,q;\alpha} = \ell^A$ if and only if one of the following conditions holds:
    \begin{enumerate}
        \item [\textup{(1)}]$1<p<\infty, q=p, \alpha\in\R$ and $A (t)\sim t^p(1-\log(t))^{\alpha p}$,
        \item [\textup{(2)}]$p=q=1, \alpha \leq0$ and $A(t) \sim t(1-\log(t))^{\alpha}$,
        \item [\textup{(3)}]$p=q=\infty, \alpha>0$ and $A(t) \sim \exp(-t^{-\frac{1}{\alpha}})$,
        \item [\textup{(4)}]$p=\infty, q<\infty, \alpha<-\frac{1}{q}$ or $p=q=\infty, \alpha\leq0$ and $A \sim 0$.
    \end{enumerate}
\end{theorem}

\begin{myproof}
    Suppose (1) holds and first let $\alpha\geq0$. Then by using Lemma~\ref{L:odhad-p<infty} and Proposition \ref{T:Charakterizacion-Orlicz-Embedings} for $a\in\mathcal{M}$ we have:
    \begin{align*}
        &\|a\|_{p,p;\alpha}^p = \sum_{n=1}^\infty (a_n^*)^p (1+\log(n))^{p\alpha} <\infty \\
        &\iff \sum_{n=1}^\infty (a_n^*)^p (1-\log(a_n^*))^{p\alpha} <\infty\\
        &\iff \|a\|_A <\infty.
    \end{align*}
    Because $\ell^A$ is BSS and $\ell^{p,p;\alpha}$ coincides with a BSS, we get $\ell^A = \ell^{p,p;\alpha}$. This fact comes from \cite[Chapter 1, Theorem 1.8]{Ben:88}.

    Now let $\alpha<0$. From Proposition~\ref{P:Dualita-L-Z}, the form of an associate space of an Orlicz space and the already proven case for $\alpha\geq0$ we get
    \begin{equation*}
        \ell^{p,p,\alpha} = \left(\ell^{p',p',-\alpha}\right)' = \left(\ell^{\widetilde{A}}\right)' = \ell^A.
    \end{equation*}

    Suppose (3) holds. Then by using Proposition \ref{T:Charakterizacion-Orlicz-Embedings} and Lemma \ref{L:Odhad-p=infty} for $a\in\mathcal{M}$ we have
    \begin{equation*}
        \|a\|_{\infty,\infty;\alpha} <\infty \iff \|a\|_A <\infty.
    \end{equation*}
    Because $\ell^{\infty,\infty;\alpha}$ coincides with a BSS, we again get $\ell^{\infty,\infty;\alpha} = \ell^A$.

    Now suppose (2) holds. Let $\alpha<0$. From Proposition \ref{P:Dualita-L-Z} and case (3) we have
    \begin{equation*}
        \ell^{1,1;\alpha} = \left(\ell^{\infty,\infty;-\alpha}\right)' = \left(\ell^{\widetilde{A}}\right)' = \ell^A.
    \end{equation*}
    If $\alpha=0$, then by using Lemma \ref{L:odhad-p<infty} and Proposition \ref{T:Charakterizacion-Orlicz-Embedings} we get $\|a\|_{1,1;0} < \infty$ iff $\|a\|_A < \infty$ for every $a\in\mathcal{M}$. Then by the same argument as above we get $\ell^{1,1;0} = \ell^A$.

    Suppose (4) holds. From Proposition \ref{P:fundmental-sequence-L-Z} we get $\varphi_{\ell^{\infty,q;\alpha}} \approx 1$, thus by Propositions \ref{P:embeding-and-fundamental-sequence} and \ref{P:Orlicz-for-t0} (ii) we have $\ell^{\infty,q;\alpha} = \ell^\infty = \ell^A$.
    
    \medskip
    
    To prove the necessity of (1)--(4) suppose there is $B$ a Young function such that $\ell^{p,q;\alpha} = \ell^B $ and the conditions do not hold.

    If $p<1$, we have from Proposition \ref{P:fundmental-sequence-L-Z} that $\varphi_{\ell^{p,q;\alpha}}(n) \approx n^\frac{1}{p}(1+\log(n))^\alpha$. Therefore we get
    \begin{equation} \label{E:fundamental-limit-infty}
        \lim_{n\to\infty} \frac{\varphi_{\ell^{B}(n)}}{n}
        \geq \lim_{n\to\infty} c\frac{\varphi_{\ell^{p,q;\alpha}}(n)}{n} 
        \geq \lim_{n\to\infty} C n^{\frac{1}{p}-1} (1+\log(n))^\alpha = \infty.
    \end{equation}
    Thus $\frac{\varphi_{\ell^B}(n)}{n}$ is not non-increasing, which is a contradiction with quasi-concavity of a fundamental sequence.

    If $p=1, \alpha>0$, we again have from Proposition \ref{P:fundmental-sequence-L-Z} $\varphi_{\ell^{1,q;\alpha}}(n) \approx n(1+\log(n))^\alpha$ and using the same argument as in \eqref{E:fundamental-limit-infty} we have a similar contradiction.

    If $p=1, \alpha\leq0, q\neq p$, then from Lemma \ref{L:L-Z-difference} we have $\ell^{1,q;\alpha} \neq \ell^{1,1;\alpha}$. From (2) we have $A$ a Young function such that $\ell^{1,1;\alpha}= \ell^A$. By using Proposition \ref{P:fundmental-sequence-L-Z} we get
    \begin{equation*}
        \varphi_{\ell^B} \approx \varphi_{\ell^{1,q;\alpha}} \approx \varphi_{\ell^{1,1;\alpha}} \approx \varphi_{\ell^A}.
    \end{equation*}
    Therefore by using Proposition \ref{T:Charakterizacion-Orlicz-Embedings} we have:
    \begin{equation*}
        \ell^{1,q;\alpha} = \ell^B = \ell^A = \ell^{1,1;\alpha},
    \end{equation*}
    giving us the contradiction.

    If $1<p<\infty, q\neq p$, then again from Lemma \ref{L:L-Z-difference} we get $\ell^{p,q;\alpha} \neq \ell^{p,p;\alpha}$. Similarly as above from (1) we have a Young function $A$ such that $\ell^{p,p;\alpha} = \ell^A$. By using Proposition \ref{P:fundmental-sequence-L-Z} we have
    \begin{equation*}
        \varphi_{\ell^B} \approx \varphi_{\ell^{p,q;\alpha}} \approx \varphi_{\ell^{p,p;\alpha}} \approx \varphi_{\ell^A}.
    \end{equation*}
    Thus from Proposition \ref{T:Charakterizacion-Orlicz-Embedings} we get
    \begin{equation*}
        \ell^{p,q;\alpha} = \ell^B = \ell^A = \ell^{p,p;\alpha},
    \end{equation*}
    which is a contradiction to Lemma \ref{L:L-Z-difference}.

    The case $p=\infty, q\neq p, \alpha> -\frac{1}{q}$ is proven using the same method, but using (3) to prove that $\ell^{\infty,\infty;\alpha}$ is an Orlicz space.

    Finally suppose that $p=\infty$, $q<\infty$ and $\alpha = -\frac{1}{q}$. From the formula of the fundamental sequence of an Orlicz space and Proposition \ref{P:fundmental-sequence-L-Z} we get
    \begin{equation*}
        \frac{1}{B^{-1}\left(\frac{1}{n}\right)} \approx (1+\log(1+\log(n)))^{-\alpha}.
    \end{equation*}
    By a simple calculation we get
    \begin{equation} \label{E:Young-approx-loglog}
        B(t) \sim \exp\left(-e^{t^{\frac{1}{\alpha}}-1}\right).
    \end{equation}
    Set $a_n = (1+\log(1+\log(n)))^{-\frac{1}{q}}$, from Proposition \ref{P:rearrangement-of-non-increasing} (i) we have $a = a^*$. Also:
    \begin{equation*}
        \|a\|_{\infty,q;\alpha}^q = \sum_{n=1}^\infty n^{-1} (1+\log(n))^{-1} (1+\log(1+\log(n)))^{-1} = \infty.
    \end{equation*}
    Let us show that there exists $\lambda>0$ such that
    \begin{equation*}
        \sum_{n=1}^\infty \exp\left(-e^{\left(\frac{|a_n|}{\lambda}\right)^{\frac{1}{\alpha}}-1}\right) <\infty.
    \end{equation*}
    Let $\lambda>1$, then we have
    \begin{align*}
        \sum_{n=1}^\infty \exp\left(-e^{\left(\frac{|a_n|}{\lambda}\right)^{\frac{1}{\alpha}}-1}\right)
        &= \sum_{n=1}^\infty \exp\left(-e^{(\lambda^{-\frac{1}{\alpha}} -1)}e^{\log(1+\log(n))}\right)\\ 
        &= \sum_{n=1}^\infty \exp\left(-e^{(\lambda^{-\frac{1}{\alpha}}-1)}(1+\log(n))\right)\\
        &= \sum_{n=1}^\infty \exp\left(-e^{(\lambda^{-\frac{1}{\alpha}} -1)}\right) n^{-e^{(\lambda^{-\frac{1}{\alpha}} -1)}} <\infty,
    \end{align*}
    because $-e^{(\lambda^{-\frac{1}{\alpha}} -1)}<-1$. Thus by \eqref{E:Young-approx-loglog} we get $a\in\ell^B$. 

    If $p,q,\alpha$ satisfy one of the conditions from (1)-(4), then from the sufficient condition and Proposition \ref{T:Charakterizacion-Orlicz-Embedings} we have $B\sim A$, which concludes the proof.
\end{myproof}

\section*{Declarations}

\textbf{Competing Interests:} The author declares no competing interests.

\textbf{Funding Information:} This research was supported in part by the Grant no.~23-04720S from the Czech Science Foundation and by the Grant no. 26-21107S from the Czech Science Foundation.

\textbf{Author contribution:} This is a single author paper.

\textbf{Data Aviability statement:} Data is entirely included in the paper.

\bibliographystyle{abbrv}
\bibliography{bibliography}

\end{document}